\documentclass[11pt,reqno]{amsart}
\usepackage[utf8]{inputenc}
\usepackage{lineno}
\usepackage[english]{varioref}
\usepackage{dcolumn}
\usepackage[english]{babel}
\usepackage{enumerate}
\usepackage{graphicx}
\usepackage{amsmath}
\usepackage{amsfonts}
\usepackage{amssymb}
\usepackage{amsthm}
\usepackage{caption}
\usepackage{marginnote}
\usepackage{stackrel,extarrows,amssymb}
\usepackage[all,cmtip]{xy}

\newlength{\xywd}
\newcommand{\xyrightarrow}[2][]{%
  \sbox{0}{$\scriptstyle#1$}%
  \xywd=\wd0
  \sbox{0}{$\scriptstyle#2$}%
  \ifdim\wd0>\xywd \xywd=\wd0 \fi
  \xymatrix@C\dimexpr\xywd+1em\relax{{}\ar[r]^{#2}_{#1}&{}}%
}

\newtheorem{theorem}{Theorem}
\newtheorem{corollary}{Corollary}[theorem]
\theoremstyle{definition}
\newtheorem{defn}{Definition}[section]
\usepackage[a4paper,margin=1.0in]{geometry}
\usepackage{tikz}
\usepackage{cleveref}
\crefname{section}{§}{§§}
\Crefname{section}{§}{§§}
\newcommand{\geqs}{\geqslant}
\newcommand{\leqs}{\leqslant}
\newcommand{\Nb}{{\mathbb N}}
\newcommand{\Rb}{{\mathbb R}}

\newcommand{\sgn}{\operatorname{sgn}} 

\makeatletter
\renewcommand*\env@matrix[1][\arraystretch]{%
  \edef\arraystretch{#1}%
  \hskip -\arraycolsep
  \let\@ifnextchar\new@ifnextchar
  \array{*\c@MaxMatrixCols c}}
\makeatother

\newtheorem{lem}[theorem]{Lemma}
\newtheorem{prop}[theorem]{Proposition}

\theoremstyle{definition}
\theoremstyle{remark}

\usepackage{lscape}

\begin{document}
\title[Safronov-Dubovski coagulation equation]{Theoretical Analysis of a discrete population balance model for Sum kernel}

\author{Sonali Kaushik} 
\address[S. Kaushik]{ BITS Pilani, Department of Mathematics,
 Pilani Campus, Rajasthan, 333031, India.}  
\email{p20180023@pilani.bits-pilani.ac.in}

\author{Rajesh Kumar} 
\address[R. Kumar]{ BITS Pilani, Department of Mathematics,
 Pilani Campus, Rajasthan, 333031, India.} 
 \email{rajesh.kumar@pilani.bits-pilani.ac.in}

\author{Fernando P. da Costa}
\address[F.P. da Costa]{Univ. Aberta, Dep. of Sciences and Technology,
  Rua da Escola Polit\'ecnica 141-7, P-1269-001 Lisboa, Portugal, and
  Univ. Lisboa, Instituto Superior T\'ecnico, Centre for Mathematical
  Analysis, Geometry and Dynamical Systems, Av. Rovisco Pais 1,
  P-1049-001 Lisboa, Portugal.}  
\email{fcosta@uab.pt}

\thanks{\emph{Corresponding author:} R. Kumar}

\date{\today}

\subjclass{Primary 34D20, 15A18; Secondary 92C50}

\keywords{Discrete population balance model; Safronov-Dubovski coagulation equation; Oort-Hulst-Safronov (OHS) equation; existence of solutions; conservation of mass; differentiability}

\begin{abstract}
The Oort-Hulst-Safronov equation, shorterned as OHS is a relevant population balance model. Its discrete form, developed by Dubovski is the main focus of our analysis. The existence and density  conservation are established for the coagulation rate $V_{i,j} \leqs (i+j),$ $\forall i,j \in \mathbb{N}$. 
Differentiability of the solutions is investigated for the kernel $V_{i,j} \leqs i^{\alpha}+j^{\alpha}$ where $0 \leqs \alpha \leqs 1$. The article finally deals with the uniqueness result that requires the boundedness of the second moment.
\end{abstract}

\maketitle

\section{Introduction}
The \emph{coagulation} is defined as the process when clusters of mass $i$ and $j$ ($i$ and $j$ mers) merge together to generate a $(i+j)$-mer. 
Coagulation processes have a plethora of real world appications, including collision of asteroids 
\cite{piotrowski1953collisions}, red blood cell aggregation \cite{perelson1984kinetics}, helium bubble formation in nuclear materials 
\cite{bonilla2006kinetics}, colloidal chemistry \cite{aizenman1979convergence}, formation of Saturn's rings \cite{brilliantov2015size}, 
among many others. 

\medskip

This paper discusses a discrete model, i.e., a model for which the properties of the particles namely size, are described by a discrete variable $i\in \Nb$, 
known as the Safronov-Dubovski (S-D) coagulation equation \cite{davidson2014existence}. The equation seems to have been first proposed by Dubovski \cite{dubovski1999atriangle} in 1999, 
long after the introduction of the continuous version, called the Oort-Hulst-Safronov (OHS) equation 
\cite{oort1946gas, safronov1972evolution},  in the context of coagulation of particles in the 
celestial phenomena. 

\medskip

The S-D model is defined for $t \in [0,\infty)$ as
\begin{equation}\label{sd}
\frac{d\psi_i(t)}{dt}=\delta_{i\geqs 2}\psi_{i-1}(t) \sum_{j=1}^{i-1} jV_{i-1,j} \psi_j(t) -\psi_i(t) \sum_{j=1}^{i}jV_{i,j} \psi_j(t) -\sum_{j=i}^{\infty} V_{i,j}\psi_i(t)\psi_j(t),\quad i \in \mathbb{N},
\end{equation} 
where $\delta_P=1$ if $P$ is true, and zero otherwise.

\medskip

The Eqn.\eqref{sd} deals with the change in concentration of $i$-clusters, and, consequently, the focus of our study is the qualitative behavior of $\psi_i(t)$ to the initial value problems defined by this system with initial conditions 
\begin{equation}\label{sdic}
\psi_i(0)=\psi_{0 i}\geqs 0.
\end{equation}

\medskip

From a physical point of view, the equations in system \eqref{sd} are the rate equations describing the time-dependent behaivour of a system of particles whose
dynamics can be described as follows: (a) particles of size $i$ are produced when a $i-1$ cluster is struck by a particle of size 
$j\leqs i-1$: the result of this collision is that the smaller $j$-cluster is pulverized into $j$ particles of size $1$, each of which attaches itself 
to different particles of size  $i-1$ to form particles of size $i$; this is described mathematically by the first expression in the right part of \eqref{sd}:
\begin{equation*}
\delta_{i\geqs 2}\psi_{i-1}(t) \sum_{j=1}^{i-1} jV_{i-1,j} \psi_j(t);
\end{equation*}
(b) particles of size $i$ are destroyed either by being impacted by smaller clusters and thus growing to clusters of size $i+1$, by the 
mechanism just described, resulting in the second term in the right hand side of \eqref{sd}
\begin{equation*}
- \psi_i(t) \sum_{j=1}^{i}jV_{i,j} \psi_j(t),
\end{equation*}
or by being themselves the smaller clusters in a collision 
with a larger $j$-cluster, in which case it is the  $i$-cluster who is pulverized into a number $i$ of $1$- clusters who will then attach to $j$-clusters 
to produce $(j+1)$-clusters, which results in the last term in the right hand side of \eqref{sd}:
\begin{equation*}
- \sum_{j=i}^{\infty} V_{i,j}\psi_i(t)\psi_j(t).
\end{equation*}
The parameters $V_{i,j}$ for $i\neq j$, called the coagulation kernel, are the rate constants for the reaction
between clusters of sizes $i$ and $j$, and are assumed to be time independent, non-negative, and symmetric, i.e, 
$V_{i,j}=V_{j,i}$. As discussed in \cite{dubovski1999atriangle}, the rate $V_{i,i}$ is equal to half of the collisions' rate 
for the particles of size $i$.

\medskip

There are various physical properties of the solution that can be investigated with the help of its moments,
which is also an extremely useful tool to handle the related mathematical problems. The $r^{\rm th}$ moment of a solution $\psi=(\psi_i)$ of \eqref{sd} is defined by 

\begin{equation}\label{rthmoment}
\mu_r(\psi(\cdot)) = \mu_r(\cdot) := \sum_{i=1}^{\infty} i^r \psi_i(\cdot).
\end{equation}

Putting $r=0$ gives the zeroth moment, denoted as $\mu_0(\cdot)$, which is the total number of particles per unit volume. Taking $r=1$ in (\ref{rthmoment}) we get the first moment, $\mu_1(\cdot)$, which can be physically interpreted as (proportional to) the mass of the system per unit volume. We expect the mass to be a conserved quantity, i.e.\,
$\mu_1(t)=\mu_1(0)$, for kernels with slowly increasing rate of coagulation. Though the physical relevance of the second 
moment has not been much discussed  in the literature, it can be interpreted as the energy dissipated in the process \cite{dubovskiui1994mathematical}. 
Bagland \cite{bagland2005convergence} established that the solution for S-D model when 
$$\lim_{j \to \infty}\frac{V_{i,j}}{j}=0, \hspace{.2 cm} i,j \geqs 1,\hspace{.2 cm}\psi_{0} \in L_1$$ exists for $t \in [0,\infty)$.
Davidson \cite{davidson2014existence} in 2014, presented
a global existence theorem, mass conservation result and uniqueness theorem for three types of kernels, namely:
$$jV_{i,j} \leqs M\; \text{for}\; j \leqs i; \quad V_{i,j} \leqs C_V h_ih_j,\; \text{with}\; \frac{h_i}{i} \to 0 \hspace{.2 cm} \text{as} \hspace{.2 cm} i \to \infty;  \quad\text{and}\quad V_{i,j} \leqs C_V \hspace{.1 cm} \forall \hspace{.1 cm} i,j, C_V \in \Rb^{+}.$$
Mass is proven to be conserved for $V_{i,j} \leq C_Vi^{1/2}j^{1/2}$ and the solution is shown to be unique in the third case, i.e., the bounded kernel is considered.   
In general, for large classes of kernels such as product kernel, mass is not a conserved quantity, see \cite{dubovski1999atriangle} for the continuous OHS equation. This phenomenon is a consequence of part of the mass of the cluster distribution $(\psi_i)$ being transported into larger and larger values of $i$,
and part of it being lost to the limit $i\to\infty$, physically interpreted as an infinite size cluster, or gel,  in a process called \emph{gelation}. One can find results on gelation for coagulation type models in, for instance, \cite{da2015mathematical}, \cite[Chapter 9]{banasiak2019analytic}. 

\medskip

In this paper, we study the existence of mass conserving solutions to the initial value problem 
\eqref{sd}-\eqref{sdic} with rate kernels satisfying $V_{i,j} \leqs (i+j),$ $\forall  i,j \in \mathbb{N}$,
and initial condition with finite mass.
To establish regularity of the solutions, we need to consider a balance between a more restrictive class of kernels 
$V_{i,j} \leqs i^\alpha + j^\alpha$, for $\alpha\in [0, 1]$ and all $i, j\in\Nb$, and initial condition with some
finite higher moment. The uniqueness result is also established for a restrictive classes of kernels, i.e., $V_{i,j} \leq C_V (i+j)$ and $V_{i,j} \leq C_V \,\min\{i^{\eta},j^{\eta}\}, 0 \leq \eta \leq 2$, $\forall i,j \in \Nb$, $C_V \in \mathbb{R}^{+}$. The boundedness of a higher moment in finite time played a significant role in proving uniqueness. Let us now define some basic notation and notions that are needed throughout.

\medskip

The set of finite mass sequences is defined by 
\begin{equation}\label{spaceX}
X=\{z=(z_k): ||z|| <\infty\},
\end{equation}
with 
\begin{equation}\label{norm1}
||z||:=\sum_{k=1}^{\infty} k|z_k|,
\end{equation}
where $(X, \|\cdot\|)$ is a Banach space. For analysis, we often consider the non-negative cone
\begin{equation}\label{coneX+}
X^{+}= \{\psi=(\psi_i) \in X: \psi_i \geqs 0\}.
\end{equation}


\begin{defn} \label{defsol}
The solution $\psi=(\psi_i)$ of the initial value problem \eqref{sd}--\eqref{sdic} on $[0,T)$ where $0<T <\infty$, is a function 
$\psi:[0,T) \to X^{+}$ with following properties 
\begin{enumerate}
\item[(a)] $\forall \hspace{.1cm} i, \hspace{.1cm} \psi_i$ is continuous.
\item[(b)] $\int_{0}^{t} \sum_{j=1}^{\infty} V_{i,j}\psi_j(s)ds <\infty$ for every $i$ and $\forall \hspace{.2cm} 0 \leq t < T$
\item[(c)] $\forall \hspace{.1cm} i$ and $\forall \hspace{.1cm} 0 \leq t <T$
\begin{equation}\label{solution}
\psi_i(t)=\psi_{0 i}+\int_{0}^{t}\Bigl(\delta_{i\geqs 2} \psi_{i-1}\sum_{j=1}^{i-1}jV_{i-1,j}\psi_j-\psi_i\sum_{j=1}^{i} jV_{i,j}\psi_j-
\psi_i\sum_{j=i}^{\infty}V_{i,j}\psi_j\Bigr)(s)ds,
\end{equation}
\end{enumerate}
where $\delta_P=1$ if $P$ is true, and is zero otherwise. 
\end{defn}

The article is organized in six
sections. The second section discusses the preliminary results required to establish the main results of the work.
Section \ref{existence} deals with the existence of solution and its corollary.
Further, in Section \ref{density conservation}, density conservation is shown for all the solutions of the given equation and the regularity result is proved in Section \ref{regularity}. Finally, the statement and  proof of the uniqueness theorem are part of Section 6. 

\section{A finite dimensional truncation}
Our general approach in this paper 
consists in considering a finite $n$-dimensional truncation
of \eqref{sd} and, after obtaining appropriate \emph{a priori} estimates for its solutions, passing to the limit $n\to\infty$
and getting corresponding results for \eqref{sd}.

\medskip

In this section we introduce a truncated system of the S-D model and study 
some useful results about the moments of its solutions.
The finite $n$-dimensional truncated system for the equation (\ref{sd}) that we shall consider
corresponds to assuming that no
particles with size larger than $n$ can exist initially or be formed by the dynamics. Thus, for the phase variable $\psi=(\psi_1, \psi_2, \ldots, \psi_n)$ the system is
\begin{equation}\label{truncation}
\frac{d\psi_i}{dt}= \Psi^{n}_i(\psi), 
\qquad \text{for $1 \leqslant i \leqslant n$,}
\end{equation}
where
\begin{align}
\Psi^{n}_1(\psi) &:= - V_{1,1}\psi_1^2-\psi_1\sum_{j=1}^{n-1}V_{1,j}\psi_j  \label{1aa} \\
\Psi^{n}_i(\psi) &:=\psi_{i-1}\sum_{j=1}^{i-1}jV_{i-1,j}\psi_j-\psi_i\sum_{j=1}^{i} jV_{i,j}\psi_j-\psi_i\sum_{j=i}^{n-1}V_{i,j}\psi_j , 
\quad \text{for $2 \leqslant i \leqslant n-1$,} \label{1a}  \\
\Psi^{n}_n(\psi) &:=\psi_{n-1}\sum_{j=1}^{n-1}jV_{n-1,j}\psi_j.  \label{1b}
\end{align}

From what was stated above the initial conditions of interest are
\begin{equation}\label{tic}
\psi_i(0)  = \psi_i^0 \geqs 0, \qquad\text{for $1\leqslant i\leqslant n$}.
\end{equation}
It can be observed here that we have truncated the last sum upto $n-1$ not $n$. This was done to make sure that the truncation conserves mass which will be beneficial in proving the existence result. The solutions to \eqref{truncation}--\eqref{tic} exists and are unique which can be proved using the fact that the right side contains polynomials and the Picard-Lindel\"of theorem. The solutions are also non-negative, established by the addition of positive $\varepsilon$ to the right-part of all equations. Now, if $\psi^\varepsilon$ satisies $\psi_i^\varepsilon(t_0)> 0$ for some $t_0 \in \mathbb{R}^{+}-\{0\}$ and 
$\psi_j^\varepsilon(t_0)=0$ for every $j\in\{ 1,\ldots, n\}$, 
then $\frac{d}{dt}\psi_j^\varepsilon(t_0)>0$. Finally, taking $\varepsilon \to 0$ gives the non-negativity (see \cite[Theorem III-4-5]{hsieh2012basic}).

\medskip

As is usual in the analysis of coagulation type systems, estimates about the time 
evolution of moments of solutions are of paramount 
importance. For the truncated system, and in a way analogous to the $r^{\text{th}}$ moments
defined in \eqref{rthmoment}, we consider the quantities
\begin{equation}
\mu^n_g(t) := \sum_{i=1}^ng_i\psi_i(t), \label{tmoment}
\end{equation}
where $g=(g_i)$ is a non-negative sequence.  The following result on the evolution of $\mu_g^n$ will be relevant:

\begin{lem}\label{lemma0}
Let $\psi=(\psi_i)_{i\in\{1, \ldots, n\}}$ be a solution of   \eqref{truncation}--\eqref{tic} 
defined in an open interval $I$ containing $0$. Let $g=(g_i)$ be a real sequence. Then
\begin{equation}
\frac{d\mu_g^n}{dt} = \sum_{i=1}^{n-1}\sum_{j=i}^{n-1}(ig_{j+1}-ig_j-g_i)V_{i,j}\psi_i\psi_j \label{momenteq}
\end{equation}
\end{lem}

\begin{proof}
By \eqref{1aa}--\eqref{1b} we can write
\begin{align*} 
\frac{d\mu_g^n}{dt} &= \sum_{i=1}^{n} g_i \Psi^{n}_i(\psi) \\ 
                                    &=  \Bigl(- g_1V_{1,1}\psi_1^2-g_1\psi_1\sum_{j=1}^{n-1}V_{1,j}\psi_j\Bigr) \, + \\
                                    & \;\;\;\; +  \sum_{i=2}^{n-1} g_i\Bigl(\psi_{i-1}\sum_{j=1}^{i-1}jV_{i-1,j}\psi_j-\psi_i\sum_{j=1}^{i} jV_{i,j}\psi_j-\psi_i\sum_{j=i}^{n-1}V_{i,j}\psi_j\Bigr)\, +  \\
                                    & \;\;\;\; + g_n\psi_{n-1}\sum_{j=1}^{n-1}jV_{n-1,j}\psi_j.
\end{align*}
Rewriting the right-hand side by collecting together the first and forth terms, the second and fifth terms, and the third and sixth
we obtain
\begin{align*} 
\frac{d\mu_g^n}{dt} &=  \sum_{i=1}^{n-1}\sum_{j=1}^{i}(g_{i+1}-g_i)jV_{i,j}\psi_i\psi_j - 
 \sum_{i=1}^{n-1}\sum_{j=i}^{n-1}g_{i}V_{i,j}\psi_i\psi_j,
\end{align*}
and now altering the order of variables in the first equation and using the symmetry of the 
rate coefficients, $V_{j,i}=V_{i,j},$
we finally conclude \eqref{momenteq}. \end{proof}

\medskip

For $g_p := (i^p)$ consider the simplified notation $\mu_p^n:=\mu^n_{g_p}.$ It is clear from \eqref{momenteq} that,
for all $t\in I$,
\begin{equation}\label{zeroth}
\frac{d\mu^{n}_0(t)}{dt} \leqs 0
\end{equation}
and thus $\mu^{n}_0(t) \leqs \mu^{n}_0(0),$ for all $t\in I\cap\{t\geqs 0\}.$ This \emph{a priori} bound
implies that non negative solutions of the truncated systems \eqref{truncation}--\eqref{1b} are globally defined
forward in time, i.e., $I\supset [0, +\infty).$
Taking $g=g_1$ in \eqref{momenteq} we immediately conclude that, for all $t$, 
\begin{equation}\label{truncatedmassconservation}
\frac{d\mu_1^{n}(t)}{dt}=0,
\end{equation}
which means that solutions to the truncated system conserve mass. For further reference, this is stated in the next lemma.

\begin{lem}\label{lemma1}
Solutions to Cauchy problems for the truncated systems \eqref{truncation}--\eqref{1b} are 
globally defined forward in time and mass conserving, i.e., satisfy
\begin{equation}\label{2}
\mu^{n}_1(t)=\mu^{n}_1(0), \quad\forall t\geqs 0.
\end{equation}
\end{lem}
For the existence proof let us consider $\nu^{n}_m(t)$ defined as in \cite{ball1990discrete} by 
\begin{equation}\label{xnmt}
\nu^{n}_m(t):= \sum_{i=m}^{n}i\psi^{n}_i(t),
\end{equation}
where $\psi^n=(\psi_1^n,\ldots, \psi_n^n)$ is a solution of the $n$-dimensional truncated system \eqref{truncation}--\eqref{1b}.
From these expressions we immediately obtain
\begin{equation}\label{dxnmt}
\frac{d\nu^{n}_m(t)}{dt}=\Big(
\sum_{i=m}^{n-1}\sum_{j=1}^{i}jV_{i,j}\psi^{n}_i \psi^{n}_j + m\psi^{n}_{m-1}\sum_{j=1}^{m-1}jV_{m-1,j}\psi^{n}_j  - 
\sum_{i=m}^{n-1}\sum_{j=i}^{n-1}iV_{i,j}\psi^{n}_i\psi^{n}_j\Big)(t).
\end{equation}
Assume now $2m<n$ and consider the function $\kappa^{n}_m(\cdot)$ defined by 
\begin{equation}\label{qnmt}
\kappa^{n}_m(t) :=\sum_{i=m}^{2m} i\psi^{n}_i+2m\sum_{i=2m+1}^{n}\psi^{n}_i,
\end{equation}
where, again, $\psi^n=(\psi_1^n,\ldots, \psi_n^n)$ is a solution of the $n$-dimensional truncated system. Then, after
a few algebraic manipulations, we get
\begin{align}\label{qnmtex}
\nonumber
\frac{d\kappa^{n}_m(t)}{dt} &= \sum_{i=m}^{2m}i\psi_i^n + 2m\sum_{i=2m+1}^n\psi_i^n \nonumber \\
&= m\psi_{m-1}^{n}(t)\sum_{j=1}^{m-1}jV_{m-1,j}\psi_j^{n}(t)+\sum_{i=m}^{2m-1}\psi^{n}_i(t)\sum_{j=1}^{i}jV_{i,j}\psi^{n}_j(t) \;-
\nonumber  \\
&\;\;\;\;\;-\sum_{i=m}^{2m}\sum_{j=i}^{n-1}i \psi^{n}_i(t)V_{i,j}\psi^{n}_j(t)-2m \sum_{i=2m+1}^{n-1}\sum_{j=i}^{n-1} V_{i,j}\psi^{n}_i(t)\psi^{n}_j(t).
\end{align}
Finally, so that we can take $n\to \infty$ (or, eventually, only on a subsequence $n_k\to\infty$), we make use of the 
following lemma.

%
\begin{lem}\label{lemma2_1} 
Take $\psi_0=(\psi_{0 i})\in X^+$ and,
for each $n\in\Nb,$ consider the point $\psi_{0}^n\in X^+$ defined 
by $\psi_{0}^n = (\psi_{01}, \psi_{02}, \ldots, \psi_{0n}, 0, 0, \ldots)$ and let it 
be identified with the point of $\Rb^n$ obtained by discarding the $j^\text{th}$ components, for $j>n.$
Let $\psi^n$ be the solution of the $n$-dimensional truncated system \eqref{truncation}--\eqref{1b} when $V_{i,j} \leq (i+j)$
with initial condition $\psi^n(0)=\psi^n_0$ such that (\ref{2}) holds, then $\psi^n$ is relatively compact in $C([0,T])$.
\end{lem}
\begin{proof}
	Using the truncated system (\ref{1aa}-\ref{1b}) and mass conservation of this system (see (\ref{2})), it can shown that $\exists$ a constant $C >0$ such that $\forall$ $n \geq i \geq 1$,
	$$\sup_{t \geq 0}\Big(\psi^{n}_i(t)+\Big|\frac{d\psi^{n}_i(t)}{dt}\Big|\Big) \leq C \mu_1(0)^2.$$
	Thus, Ascoli theorem gives the intended result.
\end{proof}

Now, we have gathered all the required information to proceed with the existence results. 

\section{Existence result for the Cauchy problem}\label{existence}
We can now prove the first main result of the paper: the existence of global solutions of the Cauchy 
problem \eqref{sd}--\eqref{sdic}.

\begin{theorem}\label{existencetheorem}
Let, $V_{i,j}$ be nonnegative, symmetric for the exchange of $i$ with $j$, and 
satisfy $V_{i,j} \leqs (i+j),$ $\forall$ $i,j,$  and let $\psi_0=(\psi_{0 i}) \in X^{+}$, $\mu_1(0) <\infty$.
Then, $\exists$ a non-negative solution of \eqref{sd}--\eqref{sdic} defined globally. 
\end{theorem} 

\begin{proof}
Let $n$ be an arbitrarily fixed positive integer and let $\psi_0^n$ be defined as in
the statement of Lemma~\ref{lemma2_1}.
As we stated above, following \eqref{tic}, the initial value problem \eqref{truncation}--\eqref{tic} has a unique
solution, $\psi^n= (\psi_i^n)_{1\leqs i\leqs n}$, which is globally defined, non-negative and, by Lemma~\ref{lemma1}, density conserving.
By defining $\psi^n_i(t)=0$ when $i>n$ we can consider $\psi^n(t)$ as an element of $X^+$, for all $t$, and thus 
\begin{equation}\label{normbound}
\|\psi^n(t)\|=\sum_{i=1}^{\infty} i\psi^{n}_i(t)=\sum_{i=1}^{n}i\psi^{n}_i(t)=
\sum_{i=1}^{n}i\psi_{0 i}^n = \sum_{i=1}^{n}i\psi_{0 i}\leqs  \sum_{i=1}^{\infty} i\psi_{0 i}=\|\psi_0\|.
\end{equation}
By Lemma~\ref{lemma2_1} and \eqref{2}
for each $i$, $\exists$ a subsequence of $\psi^n$ (not relabelled) $\&$ a function 
$\psi_i:[0,\infty) \to \mathbb{R}$, are of bounded variation on each subset of $[0,\infty)$, such that 
$\psi^{n}_i(t)$ converges to $\psi_i(t)$ as $n$ approaches $\infty$, for every $t \in \mathbb{R}^{+}$. 
Thus $\forall$ $t\geqs 0$,
\begin{equation}\label{solutionbound}
\psi_i(t) \geqs 0 \hspace{.4 cm} \text{and} \hspace{.4 cm} ||\psi(t)|| \leqs ||\psi_0||.
\end{equation}

Our goal is to prove that this limit function $\psi$ is a mild solution of the initial value problem 
\eqref{sd}--\eqref{sdic}, i.e., fulfills the
conditions in Definition~\ref{defsol}. This will be done by passing to the limit $n\to\infty$ in the integrated
version of the truncated problem \eqref{truncation}--\eqref{tic}, namely
\begin{equation}\label{int1a}
\psi_i^n(t) = 
\psi_{0 i} + \int_0^t\biggl(\psi_{i-1}^n(s)\sum_{j=1}^{i-1}jV_{i-1,j}\psi_j^n(s) - 
\psi_i^n(s)\sum_{j=1}^{i} jV_{i,j}\psi_j^n(s) -
\psi_i^n(s)\sum_{j=i}^{n-1}V_{i,j}\psi_j^n(s)\biggr)ds.
\end{equation}
To do this, and also to satisfy condition 2  in Definition~\ref{defsol}, we need to
prove that, for every fixed $i \in\mathbb{N}$, $T\geqs 0,$ and $\varepsilon >0,$ there 
exists $m$ and $N_0$, with $N_0 > m \geqs i$, such that, for all $n>N_0$,
\begin{equation}\label{xnmtint}
\int_{0}^{T} \nu^{n}_{m}(t)dt \leqs \varepsilon,
\end{equation}
where $\nu_m^n$ was defined in \eqref{xnmt}.
This can be achieved by integrating \eqref{dxnmt} in $[0, t]$ and using \eqref{qnmtex} to yield
\begin{align*}
\nu^{n}_m(t)&=\nu^{n}_m(0)+ \int_{0}^{t} \bigg(\sum_{i=m}^{n-1}\sum_{j=1}^{i}jV_{i,j}\psi^{n}_i(s)\psi^{n}_j(s)
 + m\psi^{n}_{m-1}\sum_{j=1}^{m-1}jV_{m-1,j}\psi^{n}_{j} - \\
& \hspace*{7cm}  - \sum_{i=m}^{n-1}\sum_{j=i}^{n-1}iV_{i,j}\psi^{n}_i(s)\psi^{n}_j(s)\bigg)ds\\
&=\nu^{n}_{m}(0)+\kappa^{n}_m(t) - \kappa^{n}_m(0) +\int_{0}^{t} \bigg(\psi^{n}_{2m}(s)\sum_{j=1}^{2m}jV_{2m,j}\psi^{n}_j(s)+\sum_{i=2m+1}^{n-1}\sum_{j=1}^{i}jV_{i,j}\psi^{n}_i(s)\psi^{n}_j(s) \;-\\ 
& \hspace*{4.0cm} - \sum_{i=2m+1}^{n-1}\sum_{j=i}^{n-1}iV_{i,j}\psi^{n}_i(s)\psi^{n}_j(s)
+2m\sum_{2m+1}^{n-1}\sum_{j=i}^{n-1}V_{i,j}\psi^{n}_i(s)\psi^{n}_j(s)\bigg)ds.
\end{align*}
Some algebraic manipulations of the second double sum above provides
\begin{align*}
\sum_{i=2m+1}^{n-1}\sum_{j=1}^{i}jV_{i,j}\psi^{n}_i(s)\psi^{n}_j(s)
&=\sum_{i=2m+1}^{n-1}\sum_{j=1}^{2m}jV_{i,j}\psi^{n}_i(s)\psi^{n}_j(s) + \sum_{i=2m+1}^{n-1}\sum_{j=i}^{n-1}iV_{j,i}\psi^{n}_j(s)\psi^{n}_i(s).
\end{align*}
Substituting this into the above expression for $\nu^{n}_m(t)$ gives 
\begin{align}\label{J}
 \nonumber
\nu^{n}_m(t) 
&=\nu^{n}_{m}(0)+\kappa^{n}_m(t) - \kappa^{n}_m(0) +\int_{0}^{t} 
\bigg( \sum_{i=2m}^{n-1}\sum_{j=1}^{2m}jV_{i,j}\psi^{n}_i(s)\psi^{n}_j(s)\, + \\
& \hspace*{6.0cm} +2m\sum_{i=2m+1}^{n-1}\sum_{j=i}^{n-1}V_{i,j}\psi^{n}_i(s)\psi^{n}_j(s)\bigg)ds.
\end{align}
By \eqref{normbound}, \eqref{solutionbound}, and the pointwise convergence of $\psi_i^n$ to $\psi_i$ 
we conclude that, for all $t\in [0,T],$
$\forall \varepsilon>0, \forall p> \frac{4\|\psi_0\|}{\varepsilon}, \exists N_0: \forall n>N_0,$
\[
\sum_{i=1}^\infty\left|\psi_i^n(t)-\psi_i(t)\right| = \sum_{i=1}^{p-1}\left|\psi_i^n(t)-\psi_i(t)\right| + \sum_{i=p}^\infty\left|\psi_i^n(t)-\psi_i(t)\right| < \frac{\varepsilon}{2} + \frac{2}{p}\|\psi_0\| < \varepsilon,
\]
that enables us to take $n \to \infty$ in the definition of  $\kappa^{n}_m(t)$ in (\ref{qnmt}) and yields 
\begin{equation}\label{6}
\kappa^{n}_m(t)\xyrightarrow[n\to\infty]{} \sum_{i=m}^{2m} i\psi_i(t)+2m\sum_{i=2m+1}^{\infty}\psi_i(t) 
=: \kappa_m(t) \leqs  \sum_{i=m}^{\infty}i\psi_i(t),
\end{equation}
and so $\lim_{m\to \infty}\kappa_m(t)=0$ and $|\kappa_m(t)| \leqs \|\psi_0\|,$ for all $t \in [0,T].$
Therefore,  $\forall \varepsilon >0$, $\exists M, N_0$ with $N_0>M$, such that, $\forall m>M$, $n>N_0$ and $n\geqs 2m+1$, 
\begin{equation}\label{7}
\kappa_m^n(t) \leqs \tfrac{1}{3}\varepsilon, 
\end{equation}
and \begin{equation}\label{8}
\kappa^{n}_m(0)  \leqs \tfrac{1}{3}\varepsilon.
\end{equation}

By (\ref{7}) and (\ref{8}) and using the assumption $V_{i,j}\leqs (i+j)$ we can estimate the right-hand side of \eqref{J} as 
follows (redefining) 
\begin{align*}
x^{n}_m(t) 
&\leqs \varepsilon +\int_{0}^{t} \bigg( \sum_{i=2m}^{n-1}\sum_{j=1}^{2m}j(i+j)\psi^{n}_i(s)\psi^{n}_j(s) +
2m\sum_{i=2m+1}^{n-1}\sum_{j=i}^{n-1}(i+j)\psi^{n}_i(s)\psi^{n}_j(s)\bigg)ds  \\
& \leqs \varepsilon +\int_{0}^{t} \bigg( 2\sum_{i=2m}^{n-1}\sum_{j=1}^{2m}ij \psi^{n}_i(s)\psi^{n}_j(s) +
4m\sum_{i=2m+1}^{n-1}\sum_{j=i}^{n-1}i \psi^{n}_i(s)\psi^{n}_j(s)\bigg)ds  \\
& \leqs \varepsilon +\int_{0}^{t} \bigg( 2\sum_{i=m}^{n}i \psi^{n}_i(s)\sum_{j=1}^{n}j \psi^{n}_j(s) +
4\sum_{i=m}^{n}i \psi^{n}_i(s)\sum_{j=i}^{n}j \psi^{n}_j(s)\bigg)ds \\
& \leqs \varepsilon + 6\|\psi_0\|\int_{0}^{t} \nu_m^n(s)ds.
\end{align*}

Hence, thanks to Gronwall's lemma, we get, for all $t\in [0, T],$
\begin{equation}\label{9}
\nu^{n}_m(t) \leqs k_1\varepsilon
\end{equation}
where $k_1=e^{6T}$,
which implies that $\forall\varepsilon>0, \exists M, N_0$ with $N_0>M$, 
such that, $\forall m>M$, $n>N_0$ and $n\geqs 2m+1$, 
\begin{equation}\label{int9}
\int_{0}^{t}\nu^{n}_{m}(s)ds \leqs \varepsilon k_1T,\quad\text{for all $t\in [0, T]$}.
\end{equation}
Since, $\psi^{n}_i(t)$ is point-wise convergent to $\psi_i(t)$, the above expression entails that, for all $\varepsilon>0$,
there exists $M$ such that, for all $m> M,$ we have
$$\int_{0}^{T}\sum_{i=m}^{\infty}i\psi_i(t)dt \leqs \varepsilon.$$
Hence, when $V_{i,j} \leqs (i+j)$, for all $i\geqs 1,$
\begin{equation}\label{10}
\int_{0}^{T}\sum_{j=1}^{\infty}V_{i,j}\psi_j(t)dt <\infty,
\end{equation}
thus establishing  (b) in Definition~\ref{defsol}.

\medskip

Now, for every fixed $i$, take $n>i$ sufficiently large and, for any $\ell$ such that $i < \ell < n-1$, write \eqref{int1a} as
\begin{gather*}
\bigg|\psi_{i}^{n}(t)-\psi_{i}(0) - \int_{0}^{t}\Big(\psi_{i-1}^{n}(s)\sum_{j=1}^{i-1}jV_{i-1,j}\psi^{n}_j - \psi^{n}_i(s)\sum_{j=1}^{i}jV_{i,j}\psi^{n}_j(s) - \psi^{n}_i(s)\sum_{j=i}^{\ell}V_{i,j}\psi^{n}_j(s)\Big)ds\bigg|  \\
= \psi^{n}_i(s)\int_0^t\sum_{j=\ell + 1}^{n-1}V_{i,j}\psi^{n}_j(s)ds \;\leqs \; 2\|\psi_0\|\int_0^t \nu_{\ell+1}^n(s)ds.
\end{gather*}
Thus, from \eqref{int9}, for all $\varepsilon>0$, there exists $M$ such that, for all $\ell+1>M$ and all $n$ sufficiently large, 
the right hand side can be bounded above by $2\varepsilon \|\psi_0\|k_1T.$
Considering that each sum in the left hand side has a fixed and finite number of terms, that $\psi_j^n(s)\to \psi_j(s)$ pointwise as 
$n\to \infty$, and each of the three terms inside the integral is bounded by $2\|\psi_0\|^2$, we can use the dominated convergence 
theorem and take $n\to\infty$ to conclude that, for every $\varepsilon>0$, there exists $M$ such that, for all
$\ell \geqs M,$ we have
\begin{gather*}
\bigg|\psi_{i}(t)-\psi_{i}(0) - \int_{0}^{t}\Big(\psi_{i-1}(s)\sum_{j=1}^{i-1}jV_{i-1,j}\psi_j - \psi_i(s)\sum_{j=1}^{i}jV_{i,j}\psi_j(s) - \psi_i(s)\sum_{j=i}^{\ell}V_{i,j}\psi_j(s)\Big)ds\bigg|  \\
\leqs 2\varepsilon \|\psi_0\|k_1T.
\end{gather*}
Hence, by the arbitrariness of $\varepsilon$, we can let $\ell\to\infty$ and conclude that $\psi=(\psi_i)$ satisfy \eqref{solution}, 
which completes the proof. \end{proof}

Next we establish that the subsequence $\psi^{n_k}$ of solutions to the truncated system which
converges to the solution $\psi$ of \eqref{sd}-\eqref{sdic} actually does so in the strong topology
of $X$, uniformly for $t$ in compact subsets of $[0,\infty)$.
\begin{corollary}
Let $\psi^{n_k}$ be the pointwise convergent subsequence of solutions to $(\ref{1a}- \ref{tic})$. Then, $\psi^{n_k} \to \psi$ in $X$ uniformly on compact subsets of $[0,\infty)$.
\end{corollary}
\begin{proof}
To prove this, we prove that $\psi^{n_k}_i(t) \to \psi_i(t)$ for each $i$ uniformly on the compact subsets of $[0,\infty)$. For this, let $n_k$ be $n$ and for each $I<m$ 
\begin{equation}\label{znmt1}
\xi^{n}_m(t):=e^{-t} \left[\mu^{n}_1(t)-\sum_{i=1}^{m-1} i\psi^{n}_i(t) +(2m+2)\mu^{n}_1(0)^2\right].
\end{equation}
Now, differentiating (\ref{znmt1}) gives 
\begin{align}\label{dznmt1}
\frac{d\xi^{n}_m(t)}{dt}=e^{-t}\left[\frac{d\mu^{n}_1(t)}{dt}-\sum_{i=1}^{m-1} i\frac{d\psi^{n}_i(t)}{dt}\right]-e^{-t}\left[\mu^{n}_1(t)-\sum_{i=1}^{m-1}i\psi^{n}_i(t)+(2m+2)\mu^{n}_1(0)^2\right]
\end{align}
where 
\begin{equation}
\frac{d}{dt}\sum_{i=1}^{m-1} i\psi^{n}_i(t)=\sum_{i=1}^{m-2}\sum_{j=1}^{i} jV_{i,j}\psi^{n}_{i}\psi^{n}_{j}
-\sum_{i=1}^{m-1}\sum_{j=i}^{n-1}iV_{i,j}\psi^{n}_i \psi^{n}_j-(m-1)\psi^{n}_{m-1}\sum_{j=1}^{m-1}jV_{m-1,j}\psi^{n}_j.
\end{equation}
Using the above expression , Lemma \ref{lemma1} and $V_{i,j} \leqs (i+j)$ $\forall$ $i,j$ in (\ref{dznmt1}), one can obtain
\begin{align*}
\frac{d\xi^{n}_m(t)}{dt}\leq& e^{-t}\left[\sum_{i=1}^{m-1}\sum_{j=i}^{n-1} j(i+j)\psi^{n}_i \psi^{n}_j+(m-1) \psi^{n}_{m-1}\sum_{j=1}^{m-1}j(m-1+j) \psi^{n}_j-2\mu^{n}_1(0)^2-2m\mu^{n}_1(0)^2\right].
\end{align*}
Some simplifications guarantee that
$$\frac{d\xi^{n}_m(t)}{dt} \leqs 0, \hspace{.5 cm} n \geqs m, \hspace{.2 cm} t \in [0,T].$$ 
Hence, $\xi^{n}_m(t) \to \xi_m(t)$ uniformly on compact subsets of $[0,T)$ where 
\begin{equation*}
\xi_m(t):= e^{-t} \left[\mu_1(t)-\sum_{i=1}^{m-1} i\psi_i(t) +(2m+2)\mu_1(0)^2\right].
\end{equation*}
Let, $K \subset [0,\infty)$ be compact and $t_{n} \to t$ in $K$, then
$$\lim_{n \to \infty} \|\psi^{n}(t_{n})\|=\lim_{n \to \infty} \sum_{i=1}^{\infty} i\psi^{n}_i(t_{n})=\sum_{i=1}^{\infty} i\psi_i(t)=\|\psi(t)\|$$
which ensures that $\|\psi^n\| \to \|\psi\|,$ in $C(K,X)$. 
\end{proof}


\section{All Solutions Conserve Density}\label{density conservation}
In this section we prove that, under the assumption on the rate coefficients we have been using, 
all solutions of \eqref{sd}-\eqref{sdic} conserve density.

Let $\psi=(\psi_i)\in X^+$ be a solution of  
\eqref{solution}
in $[0, T]$. Multiplying each equation in \eqref{solution} by $g_i$ and adding  from $i=1$ to $n$, we have, after some algebraic manipulations, 
for all $t\in [0, T],$
\begin{multline}\label{trunmom1}
\sum_{i=1}^ng_i\psi_i(t) - \sum_{i=1}^ng_i \psi_{0\,i}= \int_0^t\sum_{i=1}^n\sum_{j=i}^n (jg_{i+1}-jg_i-g_j)V_{i,j}\psi_i(s)\psi_j(s)ds \; \\
-\int_0^t \sum_{j=1}^n\sum_{i=n+1}^\infty g_jV_{i,j}\psi_i(s)\psi_j(s)ds - \int_0^tg_{n+1}\psi_n(s)\sum_{j=1}^njV_{n,j}\psi_j(s)ds.
\end{multline}

\medskip

We start by observing that, taking $g_i\equiv i$ in \eqref{trunmom1} we conclude that
\[
 \sum_{i=1}^n i\psi_i(t) \leqs \sum_{i=1}^ni\psi_{0\,i} \leqs \|\psi_0\|,
\]
and, as this inequality is valid for all $n$, we can take the limit as $n\to\infty$ and conclude 
the \emph{a priori\/} bound $\|\psi(t)\| \leqs \|\psi_0\|.$

We now use \eqref{trunmom1} to prove that, under the assumed conditions on $V_{i,j}$ all solutions conserve density:

\begin{theorem}\label{all}
Let $V_{i,j}\leqs (i+j)$ for all $i$ and $j$.
Let $\psi=(\psi_i)\in X^+$ be a solution of the Safronov-Dubovski equation \eqref{solution}. 
Then the total density of $\psi$ is constant.
\end{theorem}

\begin{proof}
Let $A\in \Nb$ be fixed, and consider the sequence $(g_i^A)\in \ell^\infty$ defined by
\begin{equation}\label{giA}
g_i^A = i \wedge A,
\end{equation}
where $ a \wedge b=\min\{a,b\}$.
Then
\[
jg_{i+1}^A - jg_i^A-g_j^A = \begin{cases} -A & \text{on $\{(i,j): A\leqs i\leqs j\leqs n\}$},\\
0 & \text{on $\{(i,j): 1\leqs i \leqs A-1 \;\text{and}\;  i\leqs j\leqs n\}$},\end{cases}
\]
and \eqref{trunmom1} becomes, for $n>A$,
\begin{eqnarray}
\lefteqn{\sum_{i=1}^ng_i^A\psi_i(t) - \sum_{i=1}^ng_i^A \psi_{0\,i} =}\label{sumgA} \\
& = & \!  - \int_0^t A\sum_{j=A}^n \sum_{i=A}^j V_{i,j}\psi_i(s)\psi_j(s)ds \label{sum1}  \\
& &\! -\int_0^t \biggl(\sum_{j=1}^A\sum_{i=n+1}^\infty jV_{i,j}\psi_i(s)\psi_j(s) 
+ A\sum_{j=A+1}^n\sum_{i=n+1}^\infty V_{i,j}\psi_i(s)\psi_j(s)\biggr) ds   \label{sum2}  \\
& &\!  -  \int_0^tA \psi_n(s)\sum_{j=1}^njV_{n,j}\psi_j(s)ds.\label{sum3}
\end{eqnarray}

\medskip

We first estimate the term in \eqref{sum1}:
\begin{align}
A\sum_{j=A}^n \sum_{i=A}^j V_{i,j}\psi_i\psi_j & \leqs A\sum_{j=A}^n  \psi_j\sum_{i=A}^j i \psi_i + 
                                                             A\sum_{j=A}^n  j  \psi_j\sum_{i=A}^j\psi_i \nonumber \\
                                                                      & = A\sum_{j=A}^n\frac{1}{j} j \psi_j\sum_{i=A}^j i  \psi_i + 
                                                                                  A\sum_{j=A}^n  j \psi_j\sum_{i=A}^j\frac{1}{i}i\psi_i \nonumber \\
                                                                      & \leqs 2\sum_{j=A}^n j \psi_j\sum_{i=A}^n i \psi_i \nonumber \\
                                                                      & \leqs 2\sum_{j=A}^\infty j \psi_j\sum_{i=A}^\infty i \psi_i. \label{bound1A}
\end{align}
Thus, $\psi \in X^{+}$ implies that \eqref{bound1A} converges to zero as $A\to\infty.$
Furthermore, since \eqref{bound1A} is bounded above by $2\|\psi_0\|^2$, the dominated convergence theorem implies that,
for all $\varepsilon>0$ there exists $A_0$ such that, for all $n>A\geqs A_0$ the  absolute value of  \eqref{sum1}
is smaller that $\frac{\varepsilon}{5}.$

\medskip

Consider now \eqref{sum2}. For the first double sum, observe that for $n>A,$
\begin{align}
\sum_{j=1}^A\sum_{i=n+1}^\infty jV_{i,j}\psi_i\psi_j & \leqs \sum_{j=1}^Aj\psi_j \sum_{i=n+1}^\infty i \psi_ i +
                                                                                            \sum_{j=1}^Aj^{2}\psi_j \sum_{i=n+1}^\infty \psi_ i \nonumber \\
& \leqs \|\psi_0\|\sum_{i=n+1}^\infty i \psi_ i + 
            \sum_{j=1}^Aj^{2}\psi_j \frac{1}{n+1}\sum_{i=n+1}^\infty i \psi_ i \nonumber \\
& \leqs \|\psi_0\|\sum_{i=n+1}^\infty i \psi_ i + 
            \frac{A}{A+1}\sum_{j=1}^Aj\psi_j \sum_{i=n+1}^\infty i \psi_ i \nonumber \\ 
& \leqs \|\psi_0\|\sum_{i=n+1}^\infty i \psi_ i + 
           \|\psi_0\| \sum_{i=n+1}^\infty i \psi_ i \nonumber \\ 
& \leqs 2\|\psi_0\| \sum_{i=n+1}^\infty i \psi_i.   \label{bound2A}
\end{align}
For the second double sum in \eqref{sum2} we have a similar estimate:
\begin{align}
A\sum_{j=A+1}^n\sum_{i=n+1}^\infty V_{i,j}\psi_i\psi_j & \leqs A\sum_{j=A+1}^n\psi_j \sum_{i=n+1}^\infty i \psi_ i +
                                                                                            A\sum_{j=A+1}^nj \psi_j \sum_{i=n+1}^\infty \psi_ i \nonumber \\
& \leqs \frac{A}{A+1} \sum_{j=A+1}^nj\psi_j \sum_{i=n+1}^\infty i \psi_ i +
            A\sum_{j=A+1}^nj\psi_j \frac{1}{n+1}\sum_{i=n+1}^\infty i \psi_ i \nonumber \\
& \leqs \|\psi_0\| \sum_{i=n+1}^\infty i \psi_ i + \frac{A}{n+1}\|\psi_0\|\sum_{i=n+1}^\infty i \psi_ i \nonumber \\
& \leqs 2\|\psi_0\| \sum_{i=n+1}^\infty i \psi_i.   \label{bound3A}
\end{align}
Thus, by \eqref{bound2A} and \eqref{bound3A} we conclude that 
the integrand function in \eqref{sum2} is bounded by $4\|\psi_0\|^2$ and converges pointwise to zero as $n\to\infty$,
for each fixed $A$. Hence, again by the dominated convergence theorem we conclude that, as previously, for
all $\varepsilon>0$, there exists $A_0$ such that, for all $n>A_0$, the absolute value of the integral \eqref{sum2} is 
smaller than $\frac{\varepsilon}{5}.$

\medskip

Finally, let us consider \eqref{sum3}
\begin{align}A \psi_n\sum_{j=1}^njV_{n,j}\psi_j & \leqs A \psi_n\sum_{j=1}^nj(n + j)\psi_j \nonumber \\
& \leqs 2 A n \psi_n\sum_{j=1}^nj\psi_j \nonumber \\
& \leqs 2\|\psi_0\|A n \psi_n. \label{bound4}
\end{align}
Clearly, for each fixed $A$ \eqref{bound4} converges to zero as $n\to\infty$ and it is bounded above by $A\|\psi_0\|^2,$
and so the dominated convergence theorem implies that, 
for every $\varepsilon>0$, there exits $A_0=A_0(\varepsilon)$ 
such that, for any fixed $A>A_0,$ there exists $n_0=n_0(\varepsilon, A)$
such that, for all $n>n_0\vee A$, the absolute value of \eqref{sum3} is smaller than $\frac{\varepsilon}{5}.$

\medskip

To estimate \eqref{sumgA} observe that, for every $n>A$, we can write
\begin{align*}
\left|\sum_{i=1}^ng_i^A\psi_i(t) - \sum_{i=1}^ng_i^A \psi_{0\,i}\right| & \geqs \left|\sum_{i=1}^Ai\psi_i(t) - \sum_{i=1}^Ai \psi_{0\,i}\right| - 
A\left|\sum_{i=A+1}^n \psi_i(t) - \sum_{i=A+1}^n \psi_{0\,i}\right|,
\end{align*}
and thus,
\begin{equation}\label{desigualdade}
\left|\sum_{i=1}^Ai\psi_i(t) - \sum_{i=1}^A i \psi_{0\,i}\right| \leqs \sum_{i=A+1}^\infty i \psi_i(t) + \sum_{i=A+1}^\infty i \psi_{0\,i}
+ \left|\sum_{i=1}^ng_i^A \psi_i(t) - \sum_{i=1}^ng_i^A \psi_{0\,i}\right|.
\end{equation}
Now, for every $\varepsilon>0$ there exists $A_0$ such that, for all $A>A_0$, each of the first two sums in the right hand side
of \eqref{desigualdade} can be made smaller than $\frac{\varepsilon}{5}$, and since
the estimates of \eqref{sum1}-\eqref{sum3} obtained previously allow us to have the last term
in the right hand side of \eqref{desigualdade} is smaller that $\frac{3}{5}\varepsilon$, we conclude that
\[
\forall \varepsilon >0,\, \exists A_0: \forall A>A_0,\;
\left|\sum_{i=1}^Ai\psi_i(t) - \sum_{i=1}^Ai \psi_{0\,i}\right| < \varepsilon,
\]
which proves the result.
\end{proof}


\section{Differentiability}\label{regularity}

This section is devoted in proving that the solution of the S-D model is first-order differentiable, if the rate coefficients satisfy $V_{i,j} \leqs i^\alpha + j^\alpha$ for $\alpha \in [0,1].$ This requires the boundedness of $(\alpha+1)$-moments of the solutions and an invariance result, which are proved below in Lemma \ref{lemma4} and Theorem \ref{momenttheorem}, respectively.
 
\medskip
\begin{lem}\label{lemma4}
Let the non-negative kernel 
$V_{i,j}$ satisfy $V_{i,j} \leqs i^{\alpha}+j^{\alpha},$ for all $i,j \geqs 1,$ and for some fixed $0 \leqs \alpha \leqs 1$.
For any $T\in (0, \infty)$, let $\psi$ be a solution to (1)-(2) 
 in $[0, T]$ with initial condition $\psi_0\in X^+$. If the $(\alpha +1)$-moment of $\psi_0$, $\mu_{1+\alpha}(\psi_0)$,  
is bounded, then the $(\alpha +1)$-moment of $\psi(t)$ is also bounded for all $t \in [0, T].$ 
\end{lem}

\begin{proof}
Let $\psi=(\psi_i)$ be a solution of (1)-(2) 
in $[0, T]$ with initial condition $\psi_0$. 
By (7), 
adding the components from $i=1$ to $n$, we have, for all $t\in [0, T],$
\begin{multline}\label{trunmom}
\!\!\!\sum_{i=1}^n i^{1+\alpha}\psi_i(t) + \int_0^t\sum_{i=1}^n\sum_{j=i}^\infty i^{i+\alpha}V_{i,j}\psi_i(s)\psi_j(s)ds
+\int_0^t(n+1)^{1+\alpha} \psi_n(s)\sum_{j=1}^njV_{n,j}\psi_j(s)ds  \\
= \sum_{i=1}^ni^{1+\alpha} \psi_{0\,i} + \int_0^t\sum_{i=1}^n\sum_{j=1}^i\bigl(j(i+1)^{1+\alpha}-ji^{1+\alpha}\bigr)V_{i,j}\psi_i(s)\psi_j(s)ds.
\end{multline}
Due to the non-negativity of solutions, \eqref{trunmom} implies that
\begin{equation}\label{bound1}
\sum_{i=1}^n i^{1+\alpha}\psi_i(t)  \leqs \sum_{i=1}^n i^{1+\alpha} \psi_{0\,i} + \int_0^t\sum_{i=1}^n \sum_{j=1}^i\bigl(j(i+1)^{1+\alpha}-ji^{1+\alpha}\bigr)V_{i,j}\psi_i(s)\psi_j(s)ds.
\end{equation}
Since $\alpha\in [0,1]$ and $i\geqs 1$ we have 
$(i+1)^{1+\alpha}-i^{1+\alpha} \leqs (1+\alpha)i^\alpha + \frac{(1+\alpha)\alpha}{2!}$, and so,
\[
j\bigl((i+1)^{1+\alpha}-i^{1+\alpha}\bigr)V_{i,j} \leqs (1+\alpha)ji^{2\alpha} + (1+\alpha)j^{1+\alpha}i^{\alpha} +
\frac{(1+\alpha)\alpha}{2!}ji^\alpha + \frac{(1+\alpha)\alpha}{2!}j^{1+\alpha},
\]
from where, using $\sum_{i=1}^n i\psi_i(s) \leqs \|\psi_0\|,$ we obtain
\begin{eqnarray*}
\lefteqn{\sum_{i=1}^n \sum_{j=1}^i\bigl(j(i+1)^{1+\alpha}-ji^{1+\alpha}\bigr)V_{i,j}\psi_i(s)\psi_j(s) } \\
& \leqs & \!\!\frac{1}{2}(1+\alpha)\alpha\|\psi_0\|^2 + \frac{1}{2}(1+\alpha)(4+\alpha)\|\psi_0\|\sum_{i=1}^n i^{1+\alpha}\psi_i(s)\\
& \leqs & \!\!\|\psi_0\|^2 + 5\|\psi_0\|\sum_{i=1}^n i^{1+\alpha}\psi_i(s)
\end{eqnarray*}
which, upon substitution in \eqref{bound1}, gives
\begin{equation*}\label{bound2}
\sum_{i=1}^n i^{1+\alpha}\psi_i(t)  \leqs \sum_{i=1}^n i^{1+\alpha} \psi_{0\,i} +\|\psi_0\|^2T + 
\int_0^t 5\|\psi_0\|\sum_{i=1}^n i^{1+\alpha} \psi_i(s)ds.
\end{equation*}
Hence, by Gronwall's lemma, we conclude that, for all $t\in [0, T]$ and $n\geqs 1,$
\begin{align}
\sum_{i=1}^n i^{1+\alpha}\psi_i(t) & \leqs \Bigl(\sum_{i=1}^ni^{1+\alpha} \psi_{0\,i} +T\|\psi_0\|^2\Bigr)e^{5\|\psi_0\|t}\nonumber \\
& \leqs  \bigl(\mu_{1+\alpha}(\psi_0) +T\|\psi_0\|^2\bigr)e^{5\|\psi_0\|t}, \label{bound2}
\end{align}
where the inequality \eqref{bound2} is due to the assumption about the boundedness of the $(1+\alpha)$-moment of
the initial condition $\psi_0$.
Since the right-hand side \eqref{bound2} 
does not depend on $n$ we conclude that the same is valid in the limit $n\to \infty$, which proves the result.
\end{proof}

An important result regarding the evaluation of the higher moments of the solution is analyzed here.
\begin{theorem}\label{momenttheorem}
Assume $(g_i)$ be a real valued non-negative sequence such that $g_i=O(i^{\alpha+1})$. Let $\psi$ be a solution of \eqref{sd} when $V_{i,j} \leqs i^{\alpha}+j^{\alpha}$, $\alpha \in [0,1]$ under the assumption that $\mu_{\alpha+1} (\psi_0)$ is bounded on some interval $[0, T)$, for $0<T\leqs \infty$. Let $0\leqs t_1<t_2<T$. If the following hypotheses hold \\
$(H1)$ 
	\begin{equation}\label{h2}
		\int_{t_1}^{t_2} \sum_{i=1}^{\infty} \sum_{j=1}^{i} j(g_{i+1}-g_i) V_{i,j}\psi_j(s)\psi_i(s)ds <\infty
	\end{equation}

	$(H2)$ \begin{equation}\label{h3}
		\int_{t_1}^{t_2}\sum_{i=1}^{\infty} \sum_{j=1}^{i} g_j V_{i,j}\psi_i(s)\psi_j(s)ds <\infty
	\end{equation}  
then, for every $m\in \Nb,$ 
\begin{align}
\sum_{i=m}^\infty g_i\psi_i(t_2) - \sum_{i=m}^\infty g_i\psi_i(t_1) = & 
\int_{t_1}^{t_2}\sum_{i=m}^\infty\sum_{j=1}^i(jg_{i+1}-jg_i-g_j)V_{i,j}\psi_i(s)\psi_j(s)ds \nonumber \\
& + \delta_{m\geqs 2} \int_{t_1}^{t_2}\sum_{i=m}^\infty\sum_{j=1}^{m-1}g_{j}V_{i,j}\psi_i(s)\psi_j(s)ds \label{gmoment} \\
& + \delta_{m\geqs 2} \int_{t_1}^{t_2}g_m \psi_{m-1}(s)\sum_{j=1}^{m-1}jV_{m-1,j}\psi_j(s)ds \nonumber
\end{align}
where $\delta_P = 1$ if $P$ holds, and is equal to zero otherwise.
\end{theorem}

\begin{proof}
Take positive integers $m<n$. Multiplying each equation in \eqref{solution} by $g_i$ and summing over $i$ from $m$ to $n$, we obtain 
\begin{align}
\sum_{i=m}^n g_i\psi_i(t_2) - \sum_{i=m}^n g_i\psi_i(t_1) = & 
\int_{t_1}^{t_2}\sum_{i=m}^n \sum_{j=1}^i(jg_{i+1}-jg_i-g_j)V_{i,j}\psi_i(s)\psi_j(s)ds \label{gtmoment1}  \\
& + \delta_{m\geqs 2} \int_{t_1}^{t_2}\sum_{i=m}^n\sum_{j=1}^{m-1}g_{j}V_{i,j}\psi_i(s)\psi_j(s)ds \label{gtmoment2} \\
& + \delta_{m\geqs 2} \int_{t_1}^{t_2}g_m \psi_{m-1}(s)\sum_{j=1}^{m-1}jV_{m-1,j}\psi_j(s)ds \label{gtmoment3}\\
& - \int_{t_1}^{t_2}\sum_{i=m}^n\sum_{j=n+1}^{\infty}g_{i}V_{i,j}\psi_i(s)\psi_j(s)ds \label{gtmoment4}  \\
& - \int_{t_1}^{t_2}g_{n+1}\psi_{n}(s)\sum_{j=1}^{n}jV_{n,j}\psi_j(s)ds. \label{gtmoment5}
\end{align}
We need to prove that, as $n\to\infty$, the integrals in \eqref{gtmoment4} and \eqref{gtmoment5} converge to zero,
and the other integrals converge to the corresponding ones in the right-hand side of \eqref{gmoment}.

\medskip
 
Using $(H2)$ by interchanging the order of summation and replacing $i$ for $j$ and then following $(a)$ in Definition \ref{defsol}, one can obtain
\begin{equation}\label{g}
 \lim_{n \to \infty} \int_{t_1}^{t_2} \sum_{i=m}^{n} \sum_{j=n+1}^{\infty} g_iV_{i,j}\psi_i(s)\psi_j(s)ds=0
 \end{equation}
which proves the convergence of (\ref{gtmoment4}) to zero.
With $g_i=1$ we can write (\ref{gtmoment1})-(\ref{gtmoment5}) as follows
\begin{align*}
\sum_{i=m}^n \psi_i(t_2) - \sum_{i=m}^n \psi_i(t_1) = & 
\int_{t_1}^{t_2}\sum_{i=m}^n \sum_{j=1}^i(-V_{i,j}\psi_i(s)\psi_j(s))ds 
+ \delta_{m\geqs 2} \int_{t_1}^{t_2}\sum_{i=m}^n\sum_{j=1}^{m-1}V_{i,j}\psi_i(s)\psi_j(s)ds  \\
& + \delta_{m\geqs 2} \int_{t_1}^{t_2} \psi_{m-1}(s)\sum_{j=1}^{m-1}jV_{m-1,j}\psi_j(s)ds
 - \int_{t_1}^{t_2}\sum_{i=m}^n\sum_{j=n+1}^{\infty} V_{i,j}\psi_i(s)\psi_j(s)ds  \\
& - \int_{t_1}^{t_2} \psi_{n}(s)\sum_{j=1}^{n}jV_{n,j}\psi_j(s)ds. 
\end{align*}
Further, thanks to relation $(H2)$ and the fact that $g_j=O(j^{\alpha+1})$, the last two integrals in the above expression tend to zero as $n\rightarrow \infty$. Therefore,
\begin{align}\label{b}
\nonumber
\sum_{i=m}^\infty \psi_i(t_2) - \sum_{i=m}^\infty \psi_i(t_1) = & 
\int_{t_1}^{t_2}\sum_{i=m}^\infty \sum_{j=1}^i(-V_{i,j})\psi_i(s)\psi_j(s)ds 
+ \delta_{m\geqs 2} \int_{t_1}^{t_2}\sum_{i=m}^\infty \sum_{j=1}^{m-1}V_{i,j}\psi_i(s)\psi_j(s)ds  \\
& + \delta_{m\geqs 2} \int_{t_1}^{t_2} \psi_{m-1}(s)\sum_{j=1}^{m-1}jV_{m-1,j}\psi_j(s)ds.
\end{align}
For $p=1,2$, consider,
\begin{align}
\nonumber
|g_{n+1}| \sum_{i=n+1}^{\infty} \psi_i(t_p) & \leqs C(n+1)^{\alpha+1} \sum_{i=n+1}^{\infty} \psi_i(t_p) \\
\nonumber
& \leqs C \sum_{i=n+1}^{\infty} i^{\alpha+1} \psi_i(t_p)
\end{align} 
for $C \in \Rb^{+}$ and thus Lemma \ref{lemma4} guarantees that 
\begin{equation}\label{c}
\lim_{n \to \infty} |g_{n+1}| \sum_{i=n+1}^{\infty} \psi_i(t_p) =0.
\end{equation}
Replacing $m$ by $n+1$ in (\ref{b}), multiplying both sides by $g_{n+1}$, letting $n \to \infty$, and using $(H2)$ together with (\ref{c}) confirms that  
\begin{equation}\label{d}
 \int_{t_1}^{t_2}g_{n+1} \psi_{n}(s)\sum_{j=1}^{n}jV_{n,j}\psi_j(s)ds \to 0.
\end{equation}
By Definition~\ref{defsol}, the boundedness of $\psi_i(t)$, $(H1)$ and $(H2)$, we conclude that
\begin{equation}\label{e}
\int_{t_1}^{t_2}\sum_{i=m}^n \sum_{j=1}^i(jg_{i+1}-jg_i-g_j)V_{i,j}\psi_i(s)\psi_j(s)ds \to \int_{t_1}^{t_2}\sum_{i=m}^{\infty} \sum_{j=1}^i(jg_{i+1}-jg_i-g_j)V_{i,j}\psi_i(s)\psi_j(s)ds
\end{equation}
and 
\begin{equation}\label{f}
\delta_{m\geqs 2} \int_{t_1}^{t_2}\sum_{i=m}^n\sum_{j=1}^{m-1}g_{j}V_{i,j}\psi_i(s)\psi_j(s)ds \to \delta_{m\geqs 2} \int_{t_1}^{t_2}\sum_{i=m}^{\infty} \sum_{j=1}^{m-1}g_{j}V_{i,j}\psi_i(s)\psi_j(s)ds.
\end{equation}
Thus, using Definition~\ref{defsol} together with equations (\ref{g}), (\ref{d})-(\ref{f}) and the bounded convergence theorem, the result follows.
\end{proof}

Finally, the following proposition is discussed which is essential in showing that the solution of the S-D model is first-order differentiable. 
\begin{prop}\label{ac}
	Let $\{V_{i,j}\}_{i,j \in \Nb}$ be non-negative and $V_{i,j}\leqs i^{\alpha}+j^{\alpha}$, $0 \leqs \alpha \leqs 1$. 
Let $\psi=(\psi_i)$ be a solution
 on some interval $[0,T[$, where $0 <T \leqs \infty$, of the equation (\ref{sd}) with initial condition $f_0$ and having bounded $\mu_{\alpha+1} (\psi_0)$. Then, the series $\sum_{j=1}^{i} jV_{i,j}\psi_j(t)\psi_i(t)$ and $\sum_{j=i}^{\infty} V_{i,j}\psi_j(t)\psi_i(t)$ are absolutely continuous on the compact sub-intervals of $[0,T[$.
\end{prop}
\begin{proof}
Let, $(g_i)$ satisfy the conditions in the statement of Theorem \ref{momenttheorem}. For $(H1)$ to hold, proving the boundedness of the series $\sum_{i=1}^{\infty} \sum_{j=1}^{i} j(g_{i+1}-g_i) V_{i,j}\psi_i(t)\psi_j(t)$ is enough. Using the fact that $g_{i+1}-g_i=O(i^\alpha)$, we have the following
\begin{align*}
\sum_{i=1}^{\infty} \sum_{j=1}^{i} j(g_{i+1}-g_i) V_{i,j}\psi_i\psi_j & \leqs \sum_{i=1}^{\infty} \sum_{j=1}^{i} C j i^{\alpha} V_{i,j}\psi_i\psi_j\\
	&\leqs \sum_{i=1}^{\infty} \sum_{j=1}^{i} C j i^{\alpha} (i^{\alpha}+j^{\alpha})\psi_i\psi_j 
	&\leqs \sum_{i=1}^{\infty} \sum_{j=1}^{i} C (ji^{\alpha+1}+i^{\alpha}j^{\alpha+1})\psi_i\psi_j
	\end{align*}	
for $C$ being some positive constant. Hence, by Lemma \ref{lemma4} and Section $4$, one can obtain
		\begin{equation}\label{h2proof}
		\sum_{i=1}^{\infty} \sum_{j=1}^{i} j(g_{i+1}-g_i) V_{i,j}\psi_i\psi_j \leqs 2C N_{\alpha+1}\mu_1(0) .
	\end{equation}
Thus, $(H1)$ holds true. Further, to establish the relation $(H2)$, consider the expression,
\begin{align*}
	\sum_{i=1}^{\infty} \sum_{j=1}^{i} g_{j}V_{i,j}\psi_i(s)\psi_j(s)& \leqs  \sum_{i=1}^{\infty} \sum_{j=1}^{i} C[j^{2\alpha+1}+i^{\alpha}j^{\alpha+1}]\psi_i \psi_j\\
	&  \leqs 2C N_{\alpha+1}\mu_1(0)
\end{align*} 	
which is finite by Lemma \ref{lemma4}. Therefore, all the hypothesis of Theorem \ref{momenttheorem} are satisfied for any $t_1, t_2\in [0,T)$. Hence, considering $m=1$ for $t \in [0,T)$, equation (\ref{gmoment}) implies the uniform convergence of the series 
$\sum_{i=1}^{\infty}g_i\psi_i(t)$. Since, the series $\sum_{j=1}^{i} jV_{i,j}\psi_j(t)$ is bounded by this series, as $jV_{i,j} =O(i^{\alpha+1})$ when $j<i$, we conclude the uniform convergence of $\sum_{j=1}^{i} jV_{i,j}\psi_j(t)$.
Now, the boundedness of $\psi_i(t)$ ensures the absolute continuity of $\sum_{j=1}^{i} jV_{i,j}\psi_j(t)\psi_i(t)$. 
Also, the series $\sum_{j=i}^{\infty} V_{i,j} \psi_j(t)$ is bounded by $\sum_{j=1}^{\infty} g_j \psi_j(t)$, which yields its uniform convergence. Finally, the boundedness of $\psi_i(t)$ gives the desired result.
\end{proof}
The Definition \ref{defsol} (a), hypotheses $(H1)-(H2)$ of Theorem \ref{momenttheorem} and Proposition \ref{ac} ensure that the solution $f$ is differentiable in the classical sense in $[0,T[$.

\section{Uniqueness}
	In this section, we discuss the uniqueness of solutions for the S-D model for a restrictive class of kernels. It should be mentioned here that it was not easy to deal with $V_{i,j} \leq C_V (i+j)$ $\forall i,j$, so the kernel was restricted to establish uniqueness. 
\begin{theorem}\label{uniqueness}
	Let, the kernel $V_{i,j} \leq C_V (i+j)$ and $V_{i,j} \leq C_V \,\min\{i^{\eta},j^{\eta}\}, 0 \leq \eta \leq 2$ valid $\forall i,j \in \Nb$, $C_V \in \mathbb{R}^{+}$. If the Lemma \ref{lemma4} holds, then the equations \eqref{sd}-\eqref{sdic} have a unique solution in $X^{+}$. 
\end{theorem}
\begin{proof}
	We shall use an apporach that revolves around defining a function (say $u(t)$) that is difference of two solutions of the equation \eqref{sd} (let $\psi_i$ and $\rho_i$), both satisfying the initial condition \eqref{sdic}. Furthermore, it makes use of the properties of the signum function such as
	\begin{itemize}
		\item $(P1)$ $\sgn(\mathcal{C}(t))\frac{d \mathcal{C}(t)}{dt}= \frac{d|\mathcal{C}(t)|}{dt}$,
		\\
		\item $(P2)$ $\sgn(a) \sgn(b)=\sgn(ab)$ and $|a|=a \sgn(a)$ for any real numbers $a,b$.
	\end{itemize}
	 Our goal here is to use Gronwall's lemma to the function $u(t)$ to reach at the required result.
	Define 
	\begin{equation}\label{u(t)}
	u(t):=\sum_{i=1}^{\infty} \big|\psi_i(t)-\rho_i(t)\big|=\sum_{i=1}^{\infty}|u_i(t)|.
	\end{equation}	
Using the expression of the equation \eqref{sd}, we obtain
\begin{align*} 
\frac{d u_i(t)}{dt}&=\bigg(\delta_{i\geq 2}\psi_{i-1}(t) \sum_{j=1}^{i-1}j V_{i-1,j} \psi_j(t) - \psi_i(t) \sum_{j=1}^{i} jV_{i,j}  \psi_j(t) \sum_{j=i}^{\infty} V_{i,j} \psi_i(t) \psi_j(t)\bigg)\\
&-\bigg(\delta_{i\geq 2} \rho_{i-1}(t) \sum_{j=1}^{i-1} jV_{i-1,j} \rho_j(t) - \rho_i(t) \sum_{j=1}^{i} jV_{i,j}  \rho_j(t) - \sum_{j=i}^{\infty} V_{i,j} \rho_i(t) \rho_j(t)\bigg).
\end{align*} 
 Multiplying both sides by $\sgn(u_i(t))$ and then using $(P1)$, the above equation reduces to 
\begin{align*} 
\frac{d}{dt}|u_i(t)|&=\sgn(u_i(t))\bigg(\delta_{i\geq 2}\psi_{i-1}(t) \sum_{j=1}^{i-1}j V_{i-1,j} \psi_j(t) - \psi_i(t) \sum_{j=1}^{i} jV_{i,j}  \psi_j(t) - \sum_{j=i}^{\infty} V_{i,j} \psi_i(t) \psi_j(t)\bigg)\\
&-\sgn(u_i(t))\bigg(\delta_{i\geq 2} \rho_{i-1}(t) \sum_{j=1}^{i-1} jV_{i-1,j} \rho_j(t) - \rho_i(t) \sum_{j=1}^{i} jV_{i,j}  \rho_j(t) - \sum_{j=i}^{\infty} V_{i,j} \rho_i(t) \rho_j(t)\bigg).
\end{align*} 
Further, integrating both sides, using $u(0)=0$ and summing over $i$ from $1$ to $\infty$ yield
\begin{align*} 
u(t)= \int_{0}^{t}\sum_{i=1}^{\infty} \sgn(u_i(h) \bigg(\delta_{i\geq 2}\psi_{i-1}(h) \sum_{j=1}^{i-1}j V_{i-1,j} \psi_j(h) - \psi_i(h) \sum_{j=1}^{i} jV_{i,j}  \psi_j(h) - \sum_{j=i}^{\infty} V_{i,j} \psi_i(h) \psi_j(h)\bigg) dh\\
-\int_{0}^{t} \sum_{i=1}^{\infty} \sgn(u_i(h) \bigg(\delta_{i\geq 2}\rho_{i-1}(h) \sum_{j=1}^{i-1} jV_{i-1,j} \rho_j(h) - \rho_i(h) \sum_{j=1}^{i} jV_{i,j}  \rho_j(h) - \sum_{j=i}^{\infty} V_{i,j} \rho_i(h) \rho_j(h)\bigg) dh.
\end{align*} 
Replacing $i-1$ by $i'$, then changing $i'$ by $i$ in the first and fourth sums and finally using 
$$(\psi_i\, \psi_j-\rho_i\, \rho_j)(t)=(\psi_i\,u_j+\rho_j\, u_i)(t),$$
simplify to
\begin{align*} 
u(t)&= \int_{0}^{t} \bigg( \sum_{i=1}^{\infty}\bigg( \sgn(u_{i+1})\big( \delta_{i\geq 1} \sum_{j=1}^{i}j V_{i,j} (\psi_{i}u_j+\rho_j u_i)\big) - \sgn(u_{i})\big(\sum_{j=1}^{i} jV_{i,j} 
(\psi_i\,u_j+\rho_j\, u_i)\big)\\
&\,\,\,-\sgn(u_{i}) \big(\sum_{j=i}^{\infty} V_{i,j} (\psi_i\,u_j+\rho_j\, u_i)\big)\bigg)\bigg) (h) dh.
\end{align*}
Using $(P1)$ and $(P2)$, one can obtain the following expression
\begin{align*} 
u(t)& \leq \int_{0}^{t}\bigg( \bigg(\sum_{i=1}^{\infty} \sum_{j=1}^{i}j V_{i,j} (|u_j|\, \psi_i+|u_i|\, \rho_j)\bigg)- \bigg(\sum_{i=1}^{\infty} \sum_{j=1}^{i} j V_{i,j} \big(\sgn(u_i)\, u_j\, \psi_i+|u_i|\, \rho_j\big)\bigg)\\
&\,\,\ -\bigg(\sum_{i=1}^{\infty} \sum_{j=i}^{\infty} V_{i,j} \big(\sgn(u_i)\, u_j\, \psi_i+|u_i|\, \rho_j\big)\bigg)\bigg)(h) dh.
\end{align*}
Now, cancelling second and fourth expressions and estimating the summations will give us 
\begin{align*} 
u(t) &\,\,\, \leq \int_{0}^{t}\bigg(2\sum_{i=1}^{\infty} \sum_{j=1}^{i}j V_{i,j} |u_j|\,\psi_i + \sum_{i=1}^{\infty} \sum_{j=i}^{\infty} V_{i,j}  |u_j|\, \psi_i \bigg)(h) dh.
\end{align*}
Finally, inserting the value of $V_{i,j}$, we write 
\begin{align*} 
u(t) & \leq \int_{0}^{t}\bigg(2\sum_{i=1}^{\infty} \sum_{j=1}^{i}C_{V} j (i+j) |u_j|\,\psi_i + \sum_{i=1}^{\infty} \sum_{j=i}^{\infty} C_V i^{\eta} |u_j|\, \psi_i \bigg)(h) dh.\\
\end{align*}
The application of Lemma \ref{lemma4} for $\alpha=1$ leads to 
\begin{equation*}
u(t) \leq \int_{0}^{t} (4 C_V \mu_2 + C_V \mu_2) u(h) dh.
\end{equation*}
The application of Gronwall's lemma enables us to have $u(t) \equiv 0$ which implies that $\psi_i(t) = \rho_i(t)$ $\forall$ $0 \leq t \leq T.$ Since, $T$ is arbitrary, we get uniqueness of $\psi_i(t)$.
\end{proof}

\section{Acknowledgement}
RKumar wishes to thank Science and Engineering Research Board (SERB), Department of Science and Technology (DST), India, for the funding through the project SRG/2019/001490. The research of FPdC was partially supported by Funda\c{c}\~ao para a Ci\^encia e a Tecnologia (Portugal) through project CAMGSD UID/04459/2020. The authors claim no conflict of interest.

\bibliographystyle{IEEEtran}
\bibliography{sonali}
\end{document}